\documentclass[12pt,reqno]{amsart}

\usepackage{amssymb}
\usepackage{amscd}
\usepackage{amsfonts}
\usepackage{setspace}
\usepackage{accents}
\usepackage{version}

\usepackage{graphicx}

\newtheorem{theorem}{Theorem}[section]
\newtheorem{lemma}[theorem]{Lemma}
\newtheorem{proposition}[theorem]{Proposition}
\newtheorem{corollary}[theorem]{Corollary}

\newtheorem*{mainprop5-rpt}{Proposition \ref{mainprop5}}
\newtheorem*{k5-lemma-rpt}{Lemma \ref{k5-lemma}}

\theoremstyle{definition}
\newtheorem{definition}[theorem]{Definition}
\newtheorem{question}{Question}

\renewcommand{\leq}{\leqslant}
\renewcommand{\geq}{\geqslant}

\newcommand\sk{\operatorname{sk}}
\newcommand\hcf{\operatorname{hcf}}
\def\dubtildef{\accentset{\approx}{f}(A)}
\def\F{\mathbb{F}}

\def\Z{\mathbb{Z}}
\def\E{\mathbb{E}}
\def\P{\mathbb{P}}

\def\N{\mathbb{N}}
\def\eps{\varepsilon}

\numberwithin{equation}{section}

\begin{document}

\title[Chromatic number of random Cayley graphs]{On the chromatic number of random Cayley graphs}


\author{Ben Green}
\address{The Mathematical Institute, Radcliffe Observatory Quarter, Woodstock Road, Oxford OX2 6GG}
\email{ben.green@maths.ox.ac.uk}

\thanks{The author is supported by ERC Starting Grant 274938 (Approximate Algebraic Structure) as well as by an Investigator Award from the Simons Foundation.}

\onehalfspace
\subjclass[2000]{Primary: 05C15. Secondary: 11P70. }

\begin{abstract}
Let $G$ be an abelian group of cardinality $n$,  where $\hcf(n,6) = 1$, and let $A$ be a random subset of $G$. Form a graph $\Gamma_A$ on vertex set $G$ by joining $x$ to $y$ if and only if $x + y \in A$. Then, with high probability as $n \rightarrow \infty$, the chromatic number $\chi(\Gamma_A)$ is at most $(1 + o(1))\frac{n}{2\log_2 n}$. This is asymptotically sharp when $G = \Z/n\Z$, $n$ prime.
\end{abstract}

\maketitle

\begin{center}
\emph{To B\'ela Bollob\'as on his 70th birthday}
\end{center}

\tableofcontents
\section{Introduction}

A celebrated result of Bollob\'as \cite{bollobas} asserts that a random graph from the $G(n,\frac{1}{2})$ model has chromatic number $(1 + o(1)) \frac{n}{2\log_2 n}$ whp\footnote{whp = with high probability. Throughout the paper, this will mean that the claimed statement holds with probability tending to $1$ as $n \rightarrow \infty$, sometimes with an additional constraint on $n$ which will be explicitly noted.}. Our aim in this note is to prove that $(1 + o(1)) \frac{n}{2\log_2 n}$ is an upper bound for the chromatic number of random Cayley sum graphs\footnote{It is more usual in the literature to take $A$ to be symmetric and to join $i$ to $j$ if and only if $i - j \in A$. This gives the true Cayley graph as opposed to the Cayley sum graph. There is little genuine difference between the two models, and the Cayley sum graphs are slightly simpler to handle notationally, which is why we have preferred them.} on $G$, where $G$ is an abelian group of order $n$ where $\hcf(n, 6) = 1$. If $A \subset G$ is a set then we define its \emph{Cayley sum graph} $\Gamma_A$ to be the graph on vertex set $G$ in which $i$ is joined to $j$ if and only if $i + j \in A$. 

\begin{theorem}\label{mainthm}
Let $G$ be an abelian group of order $n$, and suppose that $A \subset G$ is selected uniformly at random from all subsets of $G$. Then, whp over $n$ with $\hcf(n,6) = 1$, the chromatic number $\chi(\Gamma_A)$ of $\Gamma_A$ is at most $(1 + o(1)) \frac{n}{2\log_2 n}$.
\end{theorem}

What is meant by this is as follows: for every $\eps > 0$, the probability that $\chi(\Gamma_A) \leq (1 + \eps) \frac{n}{2\log_2 n}$ tends to $1$ as $n \rightarrow \infty$ through values of $n$ coprime to $6$.

The condition that $\hcf(n,6) = 1$ could probably be removed, but to do so would involve a number of nontrivial modifications to certain parts of the argument. This is particularly so with regard to the definition of \emph{dissociativity}, which assumes some importance later on. Since the case of greatest interest is probably $G = \Z/n\Z$, $n$ prime, we have chosen not to do this additional work here. In this case, it follows from work of Morris and the author \cite{green-morris} that we have a corresponding lower bound $\chi(\Gamma_A) \geq (1 - o(1))\frac{n}{2\log_2 n}$ whp. To see this, note that we have $\chi(\Gamma) \omega(\Gamma^c) \geq n$ for any graph $\Gamma$ on $n$ vertices, where $\Gamma^c$ denotes the complement of $\Gamma$ and $\omega$ is the clique number. Indeed if $\Gamma$ can be $k$-coloured then there is a set of vertices of size at least $n/k$, all of which get the same colour and which must therefore be independent. An independent set of vertices is the same thing as a clique in $\Gamma^c$. The stated lower bound on $\chi(\Gamma_A)$ is a consequence of this observation and the main result of \cite{green-morris}, which implies that $\omega(\Gamma^c_A) = \omega(\Gamma_{A^c}) \leq (2 + o(1))  \log_2 n$ whp.

A $k$-colouring of $\Gamma_A$ corresponds precisely to a partition $G = X_1 \cup \dots \cup X_k$ with the property that $X_i \hat{+} X_i \subset A^c$ for all $i$, where $X \hat{+} X$ denotes the restricted sumset $\{x + x' : x, x' \in X, x \neq x'\}$. Let us record the arithmetic formulation of the upper bound in Theorem \ref{mainthm} as a separate proposition.

\begin{proposition}\label{mainprop}
Suppose that $A \subset G$ is a random set and let $r = (1 + \eps) \frac{n}{2\log_2 n}$. Then whp over $n$ with $\hcf(n,6) = 1$ there is a partition $G = X_1 \cup \dots \cup X_r$ such that $X_i \hat{+} X_i \subset A$ for all $i$.
\end{proposition}

Note that we wrote $A$ instead of $A^c$, since if $A$ is a random set then so is its complement.

The reader should be aware that many of the key ideas in the proof of this proposition, and hence of Theorem \ref{mainthm}, have exact parallels in Bollob\'as's paper \cite{bollobas}, though we will not always draw attention to these explicitly. However a number of quite nontrivial technical obstacles must be overcome in this arithmetic setting, and herein lies all of the novelty of the present work. 

\emph{Previous results.} A comprehensive resource for questions concerning the clique number of random Cayley graphs is Alon's paper \cite{alon}. Alon considers different groups (not necessarily abelian) and random sets $A$ of different sizes. He notes in \cite[Theorem 2.1 (i)]{alon} that as a consequence of an earlier result of his, joint with Krivelevich and Sudakov \cite{aks}, we have $\chi(\Gamma_A) \ll \frac{n}{\log n}$ almost surely if $A \subset G$ is a random set of size $n/2$. (Here, and henceforth in the paper, $X \ll Y$ means that $X \leq CY$ for some absolute constant $C$.) The same method gives a similar bound when $A \subset G$ is selected uniformly from all subsets of $G$.

This whole argument, which is phrased in terms of graph eigenvalues, translates rather succinctly to the arithmetic setting considered here if one uses a little Fourier analysis (which amounts to essentially the same thing). We give this argument in Appendix \ref{aks-appendix}, obtaining the bound $\chi(\Gamma_A) \leq (2 + o(1)) \frac{n}{\log_2 n}$ for a random set $A$ by these methods. Note that this is four times the bound of Theorem \ref{mainthm}.

Let us note, however, that in the argument of \cite{aks} the set $A$ is only required to be \emph{pseudorandom} in the sense of having no large Fourier coefficients (the bound being weaker the larger these coefficients are). In particular it would apply when $A = Q$ is the set of quadratic residues modulo $n$ when $n$ is prime. In this case $\Gamma_Q$ is called the \emph{Paley sum graph}, and it follows from \cite{aks} that $\chi(\Gamma_Q) \leq (2 + o(1))\log_2 n$, a result that appears to be the best known for this problem. By contrast, our result gives nothing new about this specific graph.

Another observation of Alon is \cite[Proposition 4.5]{alon}, which notes a consequence of an observation of mine from \cite{green-cliquenumber}: if $G = \F_2^m$ and $n = 2^m$ then whp (as $n$ ranges over powers of $2$) $\chi(\Gamma_A) \ll \frac{n}{\log n \log \log n}$.  Thus the chromatic number of a random Cayley graph can depend on the underlying group.

\section{Preliminary man{\oe}uvres}\label{sec2}

We begin with the observation that Proposition \ref{mainprop} is implied by the following result. Here, and for the rest of the paper, we write $E[X] := X \hat{+} X$ (the letter $E$ is supposed to denote ``edge''). 

\begin{proposition}\label{mainprop2}
Suppose that $A \subset G$ is a random set. Then whp there is a partition $G = X_1 \cup \dots \cup X_r \cup X_*$ such that $|X_i| = (2 + o(1))\log_2 n$ uniformly in $i$, $E[X_i]\subset A$, and $|X_*| = o(\frac{n}{\log n})$. 
\end{proposition}

Proposition \ref{mainprop} follows upon splitting the exceptional set $X_*$ into singletons. 

Proposition \ref{mainprop2} follows in turn by a repeated application of the next statement.

\begin{proposition}\label{mainprop3}
Suppose that $A \subset G$ is a random set. Then whp every set $S \subset G$ with $|S| = \frac{n}{\log^{2} n}$ contains a set $X$ with $|X| = (2 + o(1))\log_2 n$ and $E[X] \subset A$.
\end{proposition}

\emph{Remark.} The role of quantity $\frac{n}{\log^2 n}$ here is simply to be something concrete that is $o(\frac{n}{\log n})$.

Perhaps the most obvious approach to proving Proposition \ref{mainprop3} would be to obtain a strong upper bound on the probability that there does \emph{not} exist such a set $X$ for a \emph{fixed} $S$, and then use the union bound over all $S \subset G$ of size $\frac{n}{\log^2 n}$. Unfortunately, this approach is too crude, since there are various difficulties in obtaining a good upper bound for an arbitrary set $S$. We must instead pass to a thinner class of sets $S'$ satisfying some useful technical properties. The following definition and lemma make this possible.

\begin{definition}\label{useful-def}
We say that a set $S' \subset G$ is \emph{useful} if it enjoys the following properties:
\begin{enumerate}
\item \textup{(cardinality) }
\[ 2^{-12}\frac{n}{\log^{20} n} < |S'| \leq \frac{n}{\log^{20} n};\]
\item \textup{(good clique size)}
There is some integer $k = k_{S'}$, $k  = (2 + o(1))\log_2 n$, such that 
\[ \frac{n}{2\log^8 n} < \binom{|S'|}{k} 2^{-\binom{k}{2}} \leq \frac{n}{\log^8 n}.\]
\item \textup{(lack of structure)}
For at least 90\% of all sets $X \subset S'$ of size $k$, the following is true: the number of quadruples $(x_1, x_2, s_1, s_2) \in X \times X \times S' \times S'$ with $x_1 \neq x_2$, $s_1 \neq s_2$ and $x_1 + x_2 = s_1 + s_2$ is at most $\frac{1}{\log^{15} n} |S'|$. \end{enumerate}
\end{definition}

\begin{lemma}\label{tech-lem}
Suppose that $S \subset G$ has size $\frac{n}{\log^2 n}$. Then there is a set $S' \subset S$ that is useful.
\end{lemma}
Before proceeding to the proof of this lemma, we offer some explanatory remarks concerning parts (ii) and (iii) of Definition \ref{useful-def}. Concerning (ii), let us assume (a true statement, in fact, as we shall see below in Lemma \ref{dissoc-lemma}) that ``most'' sets $X \subset S'$ of size $k$ have $|E[X]| = \binom{k}{2}$. Then $\binom{|S'|}{k} 2^{-\binom{k}{2}}$ is roughly the expected number of $X \subset S'$ with $|X| = k$ and $E[X] \subset A$.
At a later point in the argument it will be important to have an integer value of $k$ for which this number is of controlled size, slightly less than $n$ (in fact\footnote{$X \asymp Y$ means that $X$ and $Y$ have the same order of magnitude in the sense that $c_1 Y \leq X \leq c_2 Y$ for absolute constants $c_1, c_2 > 0$.} $\asymp \frac{n}{\log^8 n}$, though there is some flexibility in this choice). Unfortunately, we have 
\[ \frac{\binom{|S'|}{k} 2^{-\binom{k}{2}}}{\binom{|S'|}{k+1} 2^{-\binom{k+1}{2}}} =  \frac{(k+1) 2^k}{|S'| - k}.\] If $|S'|$ is of order $n^{1 + o(1)}$ and $k =  (2 + o(1))\log_2 n$, this will be roughly $n^{1 + o(1)}$. Hence the values of $\binom{|S'|}{k} 2^{-\binom{k}{2}}$ are rather widely spaced as $k$ varies, and unless $|S'|$ is chosen quite judiciously there will be no value of $k$ with the property we require. Item (ii) of Definition \ref{useful-def} is devoted to making just such a judicious choice. 

Concerning (iii) of Definition \ref{useful-def}, let us note that a trivial upper bound for the number of quadruples $(x_1, x_2,s_1, s_2)$ with $x_1 \neq x_2$, $s_1 \neq s_2$ and $x_1 + x_2 = s_1 + s_2$ is $O(|S'|\log^2 n)$, since any choice of $x_1, x_2, s_1$ uniquely determines $s_2$. However, this trivial bound is not necessarily a very sharp one, since we will not usually have $x_1 + x_2 - s_1 \in S'$ unless $S'$ has some particular additive structure.  Item (iii) is asserting, in a certain technical sense, that we can assume $S'$ does not have too much structure. We will only make use of (iii) at one later point in the argument, namely in the proof of Lemma \ref{conditional-count}, but it will be quite crucial there.

\emph{Remark.} The use of additively unstructured sets in contexts like this goes back at least  as far as \cite{alon-early}.

In view of the Lemma \ref{tech-lem}, to prove Proposition \ref{mainprop3} it is enough to establish the following.

\begin{proposition}\label{mainprop4}
Suppose that $A \subset G$ is a random set. Then whp every useful set $S' \subset G$ contains a set $X$ with $|X| = k_{S'}$ and $E[X] \subset A$.
\end{proposition}

We will prove this by taking a union bound over all useful sets $S'$. Since any useful set has cardinality at most $\frac{n}{\log^{20} n}$ (by item (i) of Definition \ref{useful-def}) and $\binom{n}{n/\log^{20} n} = e^{o(n/\log^{19} n)}$, it suffices to prove the following.

\begin{proposition}\label{mainprop5}
Let $n$ be sufficiently large, and let $S \subset G$ be useful in the sense of Definition \ref{useful-def}. Let $k_S = (2 + o(1))\log_2 n$ be as in Definition \ref{useful-def} \textup{(ii)}. Suppose that $A \subset G$ is a random set. Then, with probability at least $1 - e^{-n/\log^{19} n}$, there is a set $X \subset S$ with $|X| = k_{S}$ and $E[X] \subset A$.
\end{proposition}

The proof of this bound, which by the reductions just given implies our main theorem, will occupy our attention for most of the rest of the paper. To conclude this section, we prove Lemma \ref{tech-lem}.

\begin{proof}[Proof of Lemma \ref{tech-lem}] The strategy will be to choose $S' \subset S$ to be a random subset of size $\asymp n/\log^{20} n$. The lack of structure condition of Definition \ref{useful-def} (iii) will almost surely be satisfied, but (ii) will not. However, we can pass to a further subset, also of size $\asymp n/\log^{20} n$, which does enjoy this property. Doing so does not do substantial damage to (iii).

We turn to the details. Let $\eps = \frac{1}{2\log^{18} n}$ and select a set $T \subset S$ at random by picking each element of $S$ independently at random with probability $\eps$. By standard tail estimates such as \cite[Theorem A.1.4]{alon-spencer} we have
\begin{equation}\label{size} \P(| |T| - \eps |S|| \geq \textstyle\frac{1}{2} \eps |S|) < 2e^{-C \eps^2 |S|} < \frac{1}{2}.\end{equation}
By an \emph{additive quadruple} in $S$ we mean a quadruple $(s_1,s_2,s_3,s_4)$ with $s_1 \neq s_2$, $s_3 \neq s_4$ and $s_1 + s_2 = s_3 + s_4$. The number of such additive quadruples is certainly less than $|S|^3$, since $s_1,s_2,s_3$ determine $s_4$. Divide the additive quadruples into two classes: the \emph{nondegenerate} ones in which $s_1,s_2,s_3,s_4$ are all distinct, and the \emph{degenerate} ones in which either $s_1 = s_3$ or $s_1 = s_4$. There are at most $2|S|^2$ degenerate quadruples. A degenerate quadruple lies in $T$ with probability $\eps^2$, whereas a nondegenerate one lies in $T$ with probability $\eps^4$. Therefore the expected number of additive quadruples in $T$ is less than $\eps^4 |S|^3 + 2 \eps^2 |S|^2$, which is less than $2\eps^4 |S|^3$ if $N$ is large. By Markov's inequality, the number of additive quadruples in $T$ is less than $4\eps^4 |S|^3$ with probability at least $\frac{1}{2}$. Noting that the number of additive quadruples is equal to 
\[ \sum_{x, x' \in T: x \neq x'} r_{T}(x + x'),\] where $r_{T}(\xi)$ denotes the number of ways of writing $\xi$ as the sum of two \emph{distinct} elements $x,x'$ of $T$,
we see from this and \eqref{size} that there is some $T$ for which 
\begin{equation}\label{size2} \textstyle\frac{1}{2}\eps |S| \leq |T| \leq \frac{3}{2}\eps |S|\end{equation} and
\begin{equation}\label{quad}  \sum_{x, x' \in T, x \neq x'} r_{T}(x + x')  \leq 4\eps^4 |S|^3.\end{equation}
These properties will help us satisfy (i) and (iii) of Definition \ref{useful-def}. Let us leave them aside for now, and concentrate on (ii). For this we require the following lemma. 

\begin{lemma}\label{binom-tech}
Let $M$ be a sufficiently large integer. Suppose that $D$ is a real number satisfying $1 \leq D \leq M^2$. Then there is some integer $M' = M'(D)$, $2^{-10}M \leq M' \leq M$, and an integer $k = (2 + o(1))\log_2 M$ such that $D \leq \binom{M}{k} 2^{-\binom{k}{2}} \leq 2D$.
\end{lemma}
\begin{proof}
In the proof of this lemma we write $F(k,M) := \binom{M}{k} 2^{-\binom{k}{2}}$. 
First of all note that if $M > M_0(\eps)$ is sufficiently large then 
\begin{equation}\label{eq887} F( \lceil (2 - \eps) \log_2 M\rceil, M) > M^{10}\end{equation}and
\begin{equation}\label{eq888} F(\lfloor (2 + \eps) \log_2 M \rfloor, M) < M^{-10} = o(1).\end{equation}We leave the straightforward confirmation of these facts to the reader. Thus if there is a value of $k$ such that $D \leq F(k,M) \leq 2D$ with $D$ in the stated range then it automatically satisfies $k = (2 + o(1)) \log_2 M$. 

Note also that if $\log_2 M \leq k \leq 3 \log_2 M$ then 
\[ 1 \leq \frac{F(k,M)}{F(k+1, M)} = \frac{k+1}{M - k} 2^k \leq M^3.\]
Thus, in view of \eqref{eq887} and \eqref{eq888}, there is certainly some $k$ in this range such that 
\[ M^2 \leq F(k,M) \leq M^5.\]
Now we fix this $k$ and start decreasing $M$. Note that 
\[ \frac{F(k,M)}{F(k, M-1)} = \frac{M}{M-k} < 2,\] and so by decreasing $M$ one by one we do hit some $M'$ for which $D \leq F(k,M')\leq 2D$. We must give a lower bound for $M'$. To do this, note that by Lemma \ref{binom1} we have
\[ \frac{F(k, 2t)}{F(k, t)}  \geq 2^{k} > M^{1/2}\] for any $t \geq k$, 
and therefore
\[ \frac{F(k, 2^{10}t)}{F(k, t)} > M^5.\]
It follows that $M' \geq 2^{-10} M$, concluding the proof of Lemma \ref{binom-tech}.
\end{proof}
Let us return now to the proof of Lemma \ref{tech-lem}. Recall that we had isolated a subset $T \subset S$ satisfying \eqref{size2} and \eqref{quad}, where $\eps := \frac{1}{2\log^{18} n}$. By Lemma \ref{binom-tech} applied with $M = |T|$ and $D := \frac{n}{2\log^8 n}$, we can pass to a subset $S' \subset T$ which satisfies (i) and (ii) of Definition \ref{useful-def} for some $k = (2 + o(1))\log_2 n$ satisfying $D \leq \binom{M}{k}2^{-\binom{k}{2}} \leq 2D$. We claim that $S'$ also satisfies property (iii) of that lemma. 

To this end, note that \eqref{size2} and \eqref{quad} together with the bounds for $|S'|$ imply that 
\begin{equation}\label{quad-2}  \sum_{x, x' \in S', x \neq x'} r_{S' }(x+x')  \leq \sum_{x, x' \in T, x \neq x'} r_{T }(x+x')  \leq C \eps |S'|^3 \ll \frac{1}{\log^{18} n} |S'|^3.\end{equation}

Let $X \subset S'$ be a random subset of $S'$ of size $k$. Then, by \eqref{quad-2},

\begin{align*} \E_X \sum_{x, x' \in X, x \neq x'} r_{S'}(x + x') & = \sum_{x, x' \in S', x \neq x'} \P_X (x, x' \in X) r_{S'} (x + x') \\ & = (\frac{k}{|S'|})^2 \sum_{x, x' \in S', x\neq x'} r_{S' }(x+ x') \\ & \ll \big(\frac{k(k-1)}{|S'|(|S'| - 1)}\big) \cdot \frac{1}{\log^{18} n} |S'|^3  \ll \frac{1}{\log^{16} n} |S'|.\end{align*}
By Markov's inequality, we have

\[  \sum_{x, x' \in X, x \neq x'} r_{S'}(x + x') \ll \frac{1}{\log^{16} n} |S'|\] for at least 90\% of all $X \subset S'$ of size $k$, which implies (iii) of Lemma \ref{tech-lem} for large $n$.
\end{proof}

\section{The exposure martingale}

We now explain the main outline of the proof of Proposition \ref{mainprop5}. First, let us recall the statement (the reader may care to recall the definition of $S$ being \emph{useful}, which is given in Definition \ref{useful-def}, but the specifics of that definition are not important in this section).

\begin{mainprop5-rpt}
Let $n$ be sufficiently large, and let $S \subset G$ be useful. Let $k_S = (2 + o(1))\log_2 n$ be as in part \textup{(ii)} of Definition \ref{useful-def}. Suppose that $A \subset G$ is a random set. Then, with probability at least $1 - e^{-n/\log^{19} n}$, there is a set $X \subset S$ with $|X| = k_{S}$ and $E[X] \subset A$.
\end{mainprop5-rpt}

In what follows, we write $k = k_S$ for short. 

For the rest of this section our notation will be as in this proposition. 
Let $\Omega$ be the probability space consisting of all subsets of $G$, each occurring with equal probability $2^{-n}$. Thus $A$ is drawn at random from $\Omega$. Let the random variable $f : \Omega \rightarrow \N$ be defined as follows: $f(A)$ is the maximum value of $r$ for which there exist sets $X_i \subset S$, $i = 1,\dots, r$, with $|X_i| = k$ and such that the $E[X_i]$ are disjoint subsets of $A$. The task of establishing Proposition \ref{mainprop5} is equivalent to showing that 
\begin{equation}\label{to-prove-3} \P( f(A) = 0) \leq e^{-n/\log^{19} n}\end{equation} if $n$ is sufficiently large.

Let $g_1,\dots, g_n$ be some arbitrary enumeration of the elements of $G$. Let $\mathcal{F}_j$ be the sub-$\sigma$-algebra of $2^{\Omega}$ generated by sets of the form $\{A \in \Omega : 1_A(g_1) = \eps_1, \dots 1_A(g_j) = \eps_j\}$ for $(\eps_1,\dots,\eps_j) \in \{0,1\}^j$, and consider the random variables $Z_j := \E( f(A) | \mathcal{F}_j)$. We have the nesting
\[ \mathcal{F}_0 \subset \mathcal{F}_1 \subset \dots \subset \mathcal{F}_n,\] where $\mathcal{F} \subset \mathcal{F}'$ means that $\mathcal{F}'$ is a refinement of $\mathcal{F}$. The sequence $Z_j$ is a Doob martingale and we have $Z_0 = \E f(A)$ and $Z_n = f(A)$. Furthermore, flipping the value of $1_A(g_j)$ cannot change $f(A)$ by more than $1$, since if $g_j \in A$ then removing $g_j$ cannot destroy the containment $E[X_i] \subset A$ for more than one value of $i$, on account of the sets $E[X_i]$ being disjoint. It follows that the martingale $(Z_j)_{j = 0}^n$ enjoys the Lipschitz property $|Z_{j-1} - Z_j| \leq 1$ for $j = 1,2,\dots, n$, and therefore we may apply Azuma's inequality \cite[Corollary 7.2.2]{alon-spencer} to conclude that 
\[ \P( |f(A) - \E f(A)| \geq t) \leq 2 e^{-t^2/2n}.\]
In order to establish \eqref{to-prove-3}, and hence Proposition \ref{mainprop5} and our main theorem, it is enough to prove that 
\begin{equation}\label{expect} \E f(A) \geq n/\log^{9} n \end{equation}
provided that $n$ is sufficiently large.

\section{In search of edge-disjoint cliques}

In the next two sections we always assume that $n = |G|$ is sufficiently large and that $\hcf(n,6) = 1$.

Our remaining task is to prove \eqref{expect}. Let us recall the setup: we have a fixed useful set $S \subset G$, satisfying (i), (ii) and (iii) of Definition \ref{useful-def}, and we wish to give a lower bound for the expectation of $f(A)$, where $A$ ranges uniformly over all subsets of $G$. Here, $f(A)$ is the maximum value of $r$ for which there exist sets $X_i \subset S$, $i = 1,\dots, r$, with $|X_i| = k$ and such that the $E[X_i]$ are disjoint subsets of $A$. Recall that $k = k_S = (2 + o(1))\log_2 n$ is such that (ii) of Lemma \ref{tech-lem} holds.

Define $\tilde f(A)$ to be the number of sets $X \subset S$ such that 
\begin{enumerate}
\item $|X| = k$;
\item $E[X] \subset A$;
\item If $Y \subset S$ is any other set, distinct from $X$, with $|Y| = k$ and $E[Y] \subset A$ then $E[X] \cap E[Y] = \emptyset$.
\end{enumerate}

We clearly have $f(A) \geq \tilde f(A)$, and so it suffices to obtain a lower bound for $\E \tilde f(A)$. In fact, we shall consider the following technical variant of $\tilde f(A)$: define $\dubtildef$ to be the number of sets $X \subset S$ such that 

\begin{enumerate}
\item $|X| = k$;
\item $E[X] \subset A$;
\item (Lack of structure with respect to $S$) The number of quadruples $(x_1, x_2, s_1, s_2) \in X \times X \times S \times S$ with $x_1 \neq x_2$, $s_1 \neq s_2$ and $x_1 + x_2 = s_1 + s_2$ is at most $\frac{1}{\log^{15} n} |S|$;
\item (Dissociativity) If $x_1,\dots, x_4, x'_1,\dots, x'_4 \in X$ and $x_1 + x_2 + x_3 + x_4 = x'_1 + x'_2 + x'_3 + x'_4$ then $x_1,x_2,x_3,x_4$ are a permutation of $x'_1,x'_2,x'_3,x'_4$;
\item If $Y \subset S$ is any other dissociated set with $|Y| = k$ and $E[Y] \subset A$ then $E[X] \cap E[Y] = \emptyset$;
\end{enumerate}
Note that $f(A) \geq \dubtildef$, so it suffices to get a lower bound for $\E \dubtildef$. Specifically, to conclude \eqref{expect} and hence our main theorem, we need only prove the following.

\begin{proposition}\label{to-prove}
Let $S \subset G$ be a fixed useful set. Let $A \subset G$ be chosen at random, and let $\dubtildef$ be as above. Then $\E \dubtildef \geq n/\log^9 n$.
\end{proposition}
To prove this proposition we will find lower bounds for:

(1) The number of dissociated sets $X \subset S$ with $|X| = k$ and having lack of structure with respect to $S$;

(2) If $X$ is such a set, the probability that $E[X] \subset A$;

(3) The conditional probability that there is another dissociated set $Y \subset S$ with $|Y| = k$, $E[Y] \subset A$ and $E[X] \cap E[Y] \neq \emptyset$. 

Point (2) is actually rather easy: if $X$ is dissociated then all sums $x + x'$ with $x \neq x'$ are distinct apart from the trivial equalities $x + x' = x' + x$, and so $|E[X]| = \binom{k}{2}$. Therefore if $A$ is a random set then

\begin{equation}\label{2-bound} \P(E [X] \subset A) = 2^{-\binom{k}{2}}.\end{equation} We turn now to point (1), which is also quite straightforward, the point being that a random subset consisting of $k$ elements of $S$ is almost certain to be dissociated and to have lack of structure with respect to $S$. We formulate this in a lemma.

\begin{lemma}\label{dissoc-lemma}
Let $S \subset G$ be a useful and let $k = k_S$ be as above. Then at least 80\% of all sets $X \subset S$ with $|X| = k$ satisfy both lack of structure with respect to $S$ and dissociativity. 
\end{lemma}
\begin{proof}
It is enough to show that at least 90\% of sets $X \subset S$ with $|X| = k$ are dissociated, since it is part of the definition of $S$ being useful (Definition \ref{useful-def} (iii)) that at least 90\% of such sets satisfy lack of structure with respect to $S$. To prove that this is so, we select the elements of $X$ one at a time, without replacement and with an order, and ask what might happen to prevent $X$ being dissociated. If we have selected $j$ elements $\{x_1,\dots, x_j\}$ then the next element $x$ must not give rise to a nontrivial solution to such equations as $x_{i_1} + x_{i_2} + x_{i_3} + x_{i_4} = x_{i_5} + x_{i_6} + x_{i_7} + x$. The number of such equations is no more than $2^8 j^7$, and each forbids a unique  value of $x$. Another example of such an equation is $x_{i_1} + x + x + x = x_{i_2} + x_{i_3} + x_{i_4}$, which forbids a unique value of $x$ because, since $\hcf(n,6) = 1$, $G$ has no $3$-torsion\footnote{If $G$ has $2$-, $3$- or $4$- torsion then more delicate arguments, including a change in the definition of dissociativity, are required.}. Therefore if $j < k$ then the number of choices for $x_{j+1}$ is at least $|S| - 2^8 j^7 \geq |S| - 2^{15} \log^7 n$. Since $|S| \asymp \frac{n}{\log^{20} n}$ we very comfortably have 
\[ ( |S| - 2^{15} \log^7 n)^{k} \geq \textstyle\frac{9}{10} |S|^{k} \geq \frac{9}{10}k!\displaystyle \binom{|S|}{k},\] and this implies the claimed result upon dividing through by $k!$ to take account of the fact that we counted ordered $k$-tuples rather than sets.
\end{proof}

We turn now to point (3). We shall show in the next section that this conditional probability is at most $\frac{1}{2}$. This finishes the task of proving Proposition \ref{to-prove} and hence the main theorem, because by combining these estimates for (1), (2) and (3) we find that the expected number of $X$ satisfying (i) -- (v) in the definition of $\dubtildef$ above is at least
\[ 0.8 \binom{|S|}{k} \times 2^{-\binom{k}{2}} \times \frac{1}{2} \asymp \frac{n}{\log^8 n},\] this last estimate following from the ``good clique size'' property, Definition \ref{useful-def} (ii).

\section{Intersecting arithmetic cliques}

At the end of the last section we had reduced the proof of our main theorem to the following fact, which gives the required bound for the conditional probability (3).

\begin{lemma}\label{conditional-count}
Suppose that $S \subset G$ is useful and that $X \subset S$ has $|X| = k$, is dissociated, and has lack of structure with respect to $S$. Then, conditioned upon the event that $E[X] \subset A$, the probability that there is some other dissociated set $Y \subset S$ with $|Y| = k$, $E[Y] \subset A$ and $E[X] \cap E[Y] \neq \emptyset$ is at most $\frac{1}{2}$.
\end{lemma}

We will divide into a number of cases according to the value of $\ell := |E[X] \cap E[Y]|$. The case $\ell = 1$ is somewhat special, so we handle it separately. This is the only point in the argument at which the lack of structure condition, Definition \ref{useful-def} (iii), is required.

\begin{lemma}\label{conditional-count}
Suppose that $S \subset G$ is useful and that $X \subset S$ has $|X| = k$, is dissociated, and has lack of structure with respect to $S$. Then, conditioned upon the event that $E[X] \subset A$, the probability that there is some other dissociated set $Y \subset S$ with $|Y| = k$, $E[Y] \subset A$ and $|E[X] \cap E[Y]| = 1$ is at most $\frac{1}{10}$.
\end{lemma}
\begin{proof}
If $|E[X] \cap E[Y]| = 1$ then there is a pair $(x_1,x_2) \in X \times X$, $x_1 \neq x_2$, such that $x_1 + x_2 \in E[Y]$. Thus $Y$ contains two distinct elements $s_1, s_2 \in S$ with $s_1 + s_2 = x_1 + x_2$. The total number of choices for this pair of elements (across all choices of $x_1, x_2$) is bounded by $\frac{1}{\log^{15} n} |S|$, by the assumption that $X$ has lack of structure with respect to $S$. The number of choices for the remaining elements of $Y$ is at most $\binom{|S|}{k-2}$, and hence there are at most 
\[ \frac{1}{\log^{15} n} |S| \cdot \binom{|S|}{k-2} \] choices for $Y$ in total. For each such choice, the probability that $E[Y] \subset A$ given that $E[X] \subset A$ is precisely $2^{1 - \binom{k}{2}}$, since $Y$ is dissociated and hence $|E[Y]| = \binom{k}{2}$.

Therefore the probabilty we seek to bound is at most
\begin{align*} \frac{2}{\log^{15} n} |S| \cdot & \binom{|S|}{k-2} \cdot 2^{-\binom{k}{2}} \\ & = \frac{2|S|k(k-1)}{(|S| - k + 1)(|S| - k + 2) \log^{15} n} \cdot \binom{|S|}{k} 2^{-\binom{k}{2}}.\end{align*}
Recalling that $|S| \asymp \frac{n}{\log^{20} n}$, $k \asymp \log n$ and that $\binom{|S|}{k} 2^{-\binom{k}{2}} \asymp \frac{n}{\log^8 n}$ (the ``good clique size'' property, Definition \ref{useful-def} (ii)) we see that this probability is $\ll \frac{1}{\log n} = o(1)$, as required.\end{proof}

For the remainder of this section, then, we assume that $\ell = |E[X] \cap E[Y]| \geq 2$. To study this situation we introduce some further notation.

Apart from $\ell$, another key quantity will be $d$, defined as follows. Fix, once and for all, an arbitrary total ordering $\prec$ on $G$. Then any set $Y \subset S$ with $|Y| = k$ may be totally ordered as $y_1 \prec \dots \prec y_k$. Define a graph $\Gamma_Y$ on vertex set $[k]$ by joining $i$ and $j$ by an edge if and only if $y_i + y_j \in E[X]$. The number of edges in $\Gamma_Y$ is precisely $\ell$; write $d$ for the number of connected components (including isolated vertices) in $\Gamma_Y$. 

There is an important relation between $\ell$ and $d$.

\begin{lemma}\label{elldrelation}
We have $\ell \leq \frac{1}{2}(k - d + 1)(k - d)$, and $\binom{k}{2} - \ell \geq (d-1)(k - \frac{1}{2}d)$.
\end{lemma}
\begin{proof}
The second statement is obviously equivalent to the first, but we have stated it separately for convenience. Suppose that the components of $\Gamma$ have sizes (number of vertices) $x_1,\dots, x_d$, thus $x_1 + \dots + x_d = k$ and $x_1,\dots, x_d \geq 1$. Then the number of edges in $\Gamma$ is at most $\sum_i \binom{x_i}{2}$. By convexity this is largest when $x_1 = k -d +1 $ and $x_2 = \dots = x_d = 1$, in which case the number of edges is $\frac{1}{2}(k - d + 1)(k - d)$. 
\end{proof}

An important notion will be that of the \emph{skeleton} $\sk(\Gamma_Y)$ of $\Gamma_Y$. Given a graph $\Gamma$ on vertex set $[k]$, its skeleton $\sk(\Gamma) \subset \Gamma$ is a forest (union of trees) whose connected components are precisely the connected components of $\Gamma$. There are, in general, many choices of a skeleton $\sk(\Gamma)$ for each $\Gamma$, but we make an arbitrary one.

Turning back to our main task, we now dispense with the case in which $\ell$ is not extremely close to $\binom{k}{2}$. 

\begin{lemma}\label{conditional-count-2}
Suppose that $S \subset G$ is useful and that $X \subset S$ is a dissociated set with $|X| = k$. Then, conditioned upon the event that $E[X] \subset A$, the probability that there is some other dissociated set $Y \subset S$ with $|Y| = k$, $E[Y] \subset A$ and $2 \leq |E[X] \cap E[Y]| \leq \binom{k}{2} - k^{4/3}$ is at most $\frac{1}{10}$.
\end{lemma}
\begin{proof}
For each $\ell$ we enumerate the number of sets $Y$ for which $|E[X] \cap E[Y]| = \ell$ according to their skeleton $\sigma := \sk(\Gamma_Y)$. Since this skeleton has at most $\min(\ell,k)$ edges, a very crude bound for the number of choices for $\sigma$ is $\sum_{i = 0}^{\min(\ell,k)} \binom{k}{2}^{i} < k^{2\min(\ell,k)}$.

If $\sigma$ has $d$ connected components, then an upper bound for the number of choices for $Y$ is at most $\binom{|S|}{d} k^{2\min(\ell,k)}$. To see this, first select vertices $i_1,\dots,i_d \in [k]$, one in each of the $d$ connected components, and assign the values of $y_{i_j}$ arbitrarily subject to the order relation that $y_i \prec y_{i'}$ if $i < i'$. These elements must be elements of $S$, so there are $\binom{|S|}{d}$ choices. Some further vertices will be joined to an $i_j$ by an edge of $\sigma$. Suppose, for example, that $i_*$ is joined to $i_1$ by such an edge. Then $y_{i_*} + y_{i_1} \in E[X]$, and so there are at most $\binom{k}{2} < k^2$ choices for $y_{i_*}$. Similarly, some further vertices will be joined to $i_*$ by an edge of $\sigma$, and so on. By repeating this process we will eventually assign all of the values $y_1,\dots, y_k$, and the number of choices that has been made is at most the number of edges of $\sigma$, which is at most $\min(\ell, k)$.

Putting these observations together, the number of choices for $Y$, for a fixed given $\ell$, is at most $\binom{|S|}{d} k^{4\min(k,\ell)}$. For each such choice, the probability that $E[Y] \subset A$, given that $E[X] \subset A$, is $2^{- \binom{k}{2} + \ell}$. Summing over the choices of $Y$ for which $|E[X] \cap E[Y]| = \ell$, we get a total contribution of at most
\begin{equation}\label{shortly} \binom{|S|}{d} k^{4 \min(k,\ell)} \cdot 2^{-\binom{k}{2} + \ell} .\end{equation}
We handle this differently according to the size of $\ell$. 

Suppose first that $2 \leq \ell \leq k$. Then by Lemma \ref{binom3} we may bound \eqref{shortly} above by

 \[ \binom{|S|}{k} 2^{-\binom{k}{2}} \cdot 2^{\ell} k^{4\ell} \big(\frac{k}{|S| - k}\big)^{k - d}.\]
The fact that $S$ is useful tells us that $\binom{|S|}{k} 2^{-\binom{k}{2}} \asymp \frac{n}{\log^8 n} < n$. Furthermore the first relation in Lemma \ref{elldrelation} implies that $k - d \geq \sqrt{\ell}$. Since $|S| \asymp n^{1 + o(1)}$ and $k < 3\log_2 n$, these facts together allow us to bound \eqref{shortly} above by
\[ \frac{(6 \log_2 n)^{4\ell} n}{n^{(1 - o(1))\sqrt{\ell}}},\]which is $\ll n^{-c}$ for the stated range of $\ell$ (and in fact for $\ell$ up to about $\frac{k^2}{\log^{O(1)} k}$).

If $k < \ell \leq \binom{k}{2} - k^{4/3}$ then we replace $\min(k,\ell)$ by $k$, so the quantity \eqref{shortly} that we wish to bound is
\begin{equation}\label{eq46} \binom{|S|}{d} k^{4k} 2^{- \binom{k}{2} + \ell}.\end{equation}
Since $S$ is useful, we have $\binom{|S|}{k} = 2^{\binom{k}{2}} n^{1 + o(1)}$, and so from Lemma \ref{binom2} and the fact that $k \sim 2\log_2 N$ we have \[ \binom{|S|}{d}  \leq  k^d \binom{|S|}{k} ^{d/k}   \leq 2^{\frac{1}{2}d(k-1)} k^dn^{O(d/k)} \leq 2^{\frac{1}{2}d(k-1)} (2k)^{O(d)} .\]
Since $d \leq k$, an upper bound for \eqref{eq46} is therefore
\begin{equation}\label{eq47} 2^{\frac{1}{2} dk + \ell - \binom{k}{2} + O(k \log k)}.\end{equation}
Since $\ell - \binom{k}{2} \leq  -k^{4/3}$, we get an upper bound of $2^{-\frac{1}{2} k^{4/3}}$ if $d < k^{1/10}$, and this is certainly acceptable when summed over all $\ell$. If $d > k^{1/10}$, we instead apply the second bound $\binom{k}{2} - \ell \geq (d-1)(k - \frac{1}{2}d)$ from Lemma \ref{elldrelation}. This implies that \eqref{eq47} is bounded above by
\begin{equation}\label{eq48} 2^{-\frac{1}{2}d (k - d) + O(k \log k)}.\end{equation}
As above (a consequence of the first bound in Lemma \ref{elldrelation}) we have $k - d \geq \sqrt{\ell} > \sqrt{k} > k^{1/10}$, and so $k^{1/10} \leq d \leq k - k^{1/10}$. But in this range we have $d(k-d) \gg k^{11/10}$. Thus \eqref{eq48} is bounded by $2^{-c k^{11/10}}$, which is again acceptable when summed over $\ell$. This concludes the proof.
\end{proof}

Finally, we need to consider the possibility that $\binom{k}{2} - k^{4/3} < \ell \leq \binom{k}{2}$. 

\begin{lemma}\label{conditional-count-3}
Suppose that $S \subset G$ is useful and that $X \subset S$ is a dissociated set with $|X| = k$. Then, conditioned upon the event that $E[X] \subset A$, the probability that there is some other dissociated set $Y \subset S$ with $|Y| = k$, $E[Y] \subset A$ and $\binom{k}{2} - k^{4/3} <   |E[X] \cap E[Y]| \leq \binom{k}{2}$ is at most $\frac{1}{10}$.
\end{lemma}
In the regime covered here, the argument used in proving Lemma \ref{conditional-count-2} breaks down.  The key new observation here is that if $\ell = |E[X] \cap E[Y]|$ is nearly $\binom{k}{2}$ then in fact $X$ and $Y$ have substantial overlap as well, at least if $X$ and $Y$ are sufficiently dissociated. To prepare the ground for proving this, we first establish a couple of lemmas.

The first of these, which is really the key, has a tedious but basically straightforward proof which we outsource to Appendix \ref{tedious}. 

\begin{lemma}[$K_5$ lemma]\label{k5-lemma}
Suppose that $X, Z \subset G$ are dissociated, that $|Z| \geq 5$ and that $E[Z] \subset E[X]$. Then $Z \subset X$.
\end{lemma}

The next lemma is a graph-theoretic fact of a fairly standard type.

\begin{lemma}\label{turan-consequence}
Suppose that $\Gamma$ is a graph on $k$ vertices with $\ell$ edges. Then all but at most $5\sqrt{\binom{k}{2} - \ell}$ vertices of $\Gamma$ lie in a subgraph of $\Gamma$ isomorphic to the complete graph $K_5$. 
\end{lemma}
\begin{proof}
The result is trivial if $\ell = \binom{k}{2}$, so suppose $\ell \leq \binom{k}{2} - 1$. Let $V \subset [k]$ be the set of vertices not contained in any copy of $K_5$ in $\Gamma$.  If $|V| \leq 5$ then the result is immediate. Otherwise, certainly $\Gamma|_V$ does not contain any copy of $K_5$ and hence, by Tur\'an's theorem, this graph has at most $\frac{3}{8} |V|^2$ edges. Hence $\binom{k}{2} - \ell$, the number of edges in the complement of $\Gamma$, is at least $\binom{|V|}{2} - \frac{3}{8} |V|^2 \geq \frac{1}{25}|V|^2$. The result follows. 
\end{proof}

\begin{corollary}\label{corrr}
Suppose that $X, Y$ are dissociated sets with $|X| = k$ and $|E[X] \cap E[Y]| = \ell$. Then $|X \cap Y| \geq k - 5\sqrt{\binom{k}{2} - \ell}$.
\end{corollary}

\begin{corollary}\label{cor2}
Suppose that $X, Y$ are distinct dissociated sets of size $k$ and that $k$ is large. Then $|E[X] \cap E[Y]| \leq \binom{k}{2} - \frac{1}{2}k$.
\end{corollary}
\begin{proof}
Suppose the result is false. Then by Corollary \ref{corrr} we have $|X \cap Y| \geq k - \frac{5}{\sqrt{2}}\sqrt{k} > k - 4\sqrt{k}$. Since $X$ and $Y$ are distinct and have the same size $k$, there is some $y$ that lies in $Y$ but not in $X$. By Lemma \ref{k5-lemma}, $y$ is not joined in $\Gamma_Y$ to more than $3$ of the vertices corresponding to $X \cap Y$ by an edge. Indeed if it was joined to $4$ such vertices $y_1,y_2,y_3,y_4$ then we could apply Lemma \ref{k5-lemma} with $Z := \{y_1,y_2,y_3,y_4,y\}$, concluding that $Z \subset X$ and in particular that $y \in X$, contrary to assumption. It follows that the degree of $y$ in $\Gamma_Y$ is no more than $3 + 4\sqrt{k}$. Thus the complement of $\Gamma_Y$ contains at least $k - 1 - (3 + 4\sqrt{k}) > \frac{1}{2}k$ edges (if $n$, and hence $k$, is large enough).
\end{proof}

\begin{proof}[Proof of Lemma \ref{conditional-count-3}] In the statement of Lemma \ref{conditional-count-3} we assumed that $\binom{k}{2} - k^{4/3} < \ell \leq \binom{k}{2}$, but Corollary \ref{cor2} in fact allows us to assume the stronger upper bound
$\ell \leq \binom{k}{2} - \frac{1}{2} k$.

Recall that the graph $\Gamma_Y$ is defined as follows. It is a graph on vertex set $[k]$, with $i$ joined to $j$ by an edge if and only if $y_i + y_j \in E[X]$. Recall also that $d$ is the number of connected components of $\Gamma_Y$. 

We claim that the number of choices of $Y$ is at most $\binom{|S|}{d-1} 2^{o(k)}$. To see this, we consider a variant of the skeleton of $\Gamma_Y$, which we will again call $\sigma$. By Corollary \ref{corrr}, the graph $\Gamma_Y$ has one extremely large component containing a clique $\Omega$ of size at least $k - 5k^{2/3}$; every element of $Y$ assigned to a vertex in $\Omega$ is an element of $X$. Take $\sigma$ to be any collection of $\leq 5k^{2/3}$ edges such that the edges of $\Omega$ and $\sigma$ span all the connected components of $\Gamma_Y$. Such a collection $\sigma$ may be found using a greedy algorithm. The number of choices for $\sigma$ is clearly at most $\binom{k}{2}^{5k^{2/3}} = 2^{o(k)}$, and the number of choices for $\Omega$ is also $2^{o(k)}$.

Given $\Omega$ and $\sigma$, we must assign the set $Y$. All the vertices in $\Omega$ must be elements of $X$, so the number of choices for these vertices is at most the number of subsets of $X$ of size at least $k - 5k^{2/3}$, which is again $2^{o(k)}$. In each of the other $d-1$ connected components, select one vertex. The values of these vertices must be elements of $S$, and they must obey the order relation that $y_i \prec y_j$ if $i < j$, so this gives at most $\binom{|S|}{d-1}$ choices. 

The remaining unassigned vertices are connected to vertices already assigned (that it, to vertices in $\Omega$ or to the $d-1$ special vertices) by paths in $\sigma$. Each time we take an edge of $\sigma$ from an assigned vertex $i$ to a currently unassigned one $j$, the fact that $y_i + y_j \in E[X]$ gives us at most $\binom{k}{2}$ choices for $y_j$. Therefore the number of possible assignments of the remaining vertices, of which there are $k - |\Omega| - (d - 1) \leq 5k^{2/3}$, is at most $\binom{k}{2}^{5k^{2/3}} = 2^{o(k)}$, thereby concluding the proof of the claim.

Now from Lemma \ref{elldrelation} we have
\begin{equation}\label{eq994} (d - 1)(k - \frac{1}{2} d) \leq \binom{k}{2} - \ell < k^{4/3};\end{equation} since $d \leq k$, this immediately leads to the stronger bound \begin{equation}\label{eq995} d \ll k^{1/3}.\end{equation}
Equations \eqref{eq994} and \eqref{eq995} together imply that
\begin{equation}\label{eq996} \frac{1}{2}(d-1) k \leq (\frac{1}{2} + o(1)) (\binom{k}{2} - \ell). \end{equation}
Since $S$ is useful, we have 
\[ \binom{|S|}{k} 2^{-\binom{k}{2}} \leq n\] and so by Lemma \ref{binom2} we have
\begin{align*} \binom{|S|}{d-1} \leq  k^{d-1}  & \binom{|S|}{k}^{\frac{d-1}{k}} \leq 2^{\frac{1}{2} (d-1)k} k^{d-1} n^{\frac{d-1}{k}}  \\ & \leq 2^{\frac{1}{2} (d-1)k} k^{O(d)} \leq  2^{(\frac{1}{2} + o(1))(\binom{k}{2} - \ell)} k^d, \end{align*}
by \eqref{eq996}. By \eqref{eq995}, this is
\[ 2^{(\frac{1}{2} + o(1))(\binom{k}{2} - \ell) + o(k)}.\]
It follows from this and the earlier claim about the number of choices of $Y$ that the probability that there is some $Y \subset S$, distinct from $X$, such that $|E[X] \cap E[Y]| = \ell$ is bounded by
\[  2^{-(\frac{1}{2} - o(1))(\binom{k}{2} - \ell)  + o(k)}.\]  
When summed over the range $\binom{k}{2} - k^{4/3} < \ell \leq \binom{k}{2} - \frac{1}{2}k$ this is $o(1)$, as required.\end{proof}

\section{Further questions}

For a wide selection of further questions we refer the reader to \cite{alon} or \cite{christophides}. Here are two further questions:

\begin{question}
What is an asymptotic for $\chi(\Gamma_A)$, almost surely, when $A \subset \F_2^m$ is selected at random?
\end{question}

\begin{question}
If $G$ is an abelian group of size $n$ and if $A \subset G$ is selected at random, is $\chi(\Gamma_A) \omega(\Gamma_A) = (1 + o(1)) n$ almost surely?
\end{question}

\appendix

\section{Proof of the $K_5$ lemma}\label{tedious}

Let us begin by recalling the statement of the $K_5$ lemma, Lemma \ref{k5-lemma}.

\begin{k5-lemma-rpt}
Suppose that $X, Z \subset G$ are dissociated, that $|Z| \geq 5$ and that $E[Z] \subset E[X]$. Then $Z \subset X$.
\end{k5-lemma-rpt}

\begin{proof}
The proof of this is somewhat tedious, though straightforward.  We begin by looking at sets $Z$ of size $4$ for which $E[Z] \subset E[X]$. We claim that such sets are of two types: \emph{type I} in which $Z \subset X$, and \emph{type II} in which $z_i = g - x_i$ for $i = 1,2,3,4$, where $2g= x_1 + x_2 + x_3 + x_4$ and the $x_i$ are all elements of $X$.

Let us prove this claim. Suppose that $Z = \{z_1,z_2,z_3,z_4\}$ and that $z_i + z_j = x_{ij} + x'_{ij}$. We have the relations
\begin{equation}\label{master} x_{12} + x'_{12} + x_{34} + x'_{34} = x_{13} + x'_{13} + x_{24} + x'_{24} = x_{14} + x'_{14} + x_{23} + x'_{23}.\end{equation}
By dissociativity, $x_{12} \in \{x_{14}, x'_{14}, x_{23}, x'_{23}\}$, and there is no loss of generality in assuming that $x_{12} = x_{14}$.  Thus
\[ x'_{12} + x_{34} + x'_{34} = x'_{14} + x_{23} + x'_{23}.\] By dissociativity, $x'_{12} \in \{x'_{14}, x_{23}, x'_{23}\}$.
Now $z_2 \neq z_4$ and so $x'_{14} - x'_{12} = z_4 - z_2 \neq 0$. Hence we may assume without loss of generality that $x'_{12} = x'_{23}$. 
Dissociativity and \eqref{master} also implies that $x_{12} \in \{ x_{13}, x'_{13}, x_{24}, x'_{24}\}$. There are two essentially different cases: \emph{case 1} in which $x_{12} = x_{13}$, and \emph{case 2} in which $x_{12} = x_{24}$. 

Suppose we are in \emph{case 1}. Then \eqref{master} implies that $x'_{12} + x_{34} + x'_{34} = x'_{13} + x_{24} + x'_{24}$, and so by dissociativity $x'_{12} \in \{x'_{13}, x_{24}, x'_{24}\}$. Since $z_2 \neq z_3$, we cannot have $x'_{12} = x'_{13}$. Without loss of generality, then, $x'_{12} = x'_{24}$. Referring back to \eqref{master}, we see that $x_{34} + x'_{34} = x_{24} + x'_{13}  =  x_{23} + x'_{14}$. By dissociativity we have $x'_{13} \in \{x_{23}, x'_{14}\}$. However, if $x'_{13} = x'_{14}$ then we would have $z_3 = z_4$, a contradiction, and therefore $x'_{13} = x_{23}$ and $x'_{14} = x_{24}$. Writing $x_1 = x_{12}$, $x_2 = x'_{12}$, $x_3 = x_{23}$, $x_4 = x_{24}$, the above relations imply that $x_i + x_j = z_i + z_j$ for all distinct $i, j \in \{1,2,3,4\}$. Writing $w_i := x_i - z_i$, it follows that $w_i + w_j = 0$ for distinct $i, j$. This immediately implies that all of the $w_i$ are zero. This is the \emph{type I} situation.

Suppose now that we are in \emph{case 2}, that is to say $x_{12} = x_{24}$. Now \eqref{master} implies that $x_{23} + x'_{14} = x_{34} + x'_{34}$. By dissociativity we have $\{x_{23}, x'_{14} \}= \{x_{34}, x'_{34}\}$, and there is no loss of generality in assuming that $x_{23} = x_{34}$ and $x'_{14} = x'_{34}$. We also have, from \eqref{master}, $x'_{24} + x_{13} + x'_{13} = x'_{12} + x_{23} + x'_{14}$ and so, by dissociativity, $x'_{24} \in \{x'_{12}, x_{23}, x'_{14}\}$. Since the $z_i$ are all distinct we cannot have either $x'_{24} = x'_{12}$ or $x'_{24} = x_{23}$, so we must have $x'_{24} = x'_{14}$. Writing $x_1 = x_{23}$, $x_2 = x'_{14}$, $x_3 = x_{12}$ and $x_4 = x'_{12}$, we have $z_i + z_j + x_i + x_j = s$ for all distinct $i,j \in \{1,2,3,4\}$, where $s = x_1 + x_2 + x_3 + x_4$. Writing $w_i := x_i + z_i$, this implies that $w_i + w_j = s$ for all $i \neq j$. This easily implies that all of the $w_i$ are equal to some $g$ and that $2g = s$. This is the \emph{type II} situation. This completes the proof of the claim.

Suppose now that $|Z| \geq 5$ and that $E[Z] \subset E[X]$. Every set $Z' \subset Z$ with $|Z'| = 4$ is of type I or II. We claim that if there is \emph{any} set of type I then in fact the whole of $Z$ is a subset of $X$, thereby concluding the proof. Suppose that $Z'$ is of type I, thus $Z' = \{x_1, x_2, x_3, x_4\}$ with the $x_i$ being elements of $X$. Suppose that $z \in Z \setminus Z'$. Then we have $z + x_i  = x'_i + x''_i$ for $i = 1,2,3,4$ and some $x'_i, x''_i$. In particular $x_1 + x'_2 + x''_2 = x_2 + x'_1 + x''_1$, and so by dissociativity we have, without loss of generality, $x_1 = x'_1$. It then follows immediately that $z = x''_1$, thereby establishing the claim.

It remains to examine the possibility that every subset $Z' \subset Z$ with $|Z'| = 4$ is of type II. We claim this case cannot occur. We may suppose that $Z'$ consists of elements $g - x_i$, $i = 1,2,3,4$, where $2g = x_1 + x_2 + x_3 + x_4$. Let $Z'' \subset Z$ be a different subset of size $4$, intersecting $Z'$ in $\{g - x_1, g - x_2, g - x_3\}$. Since this set is also of type II, we may label it so that it consists of elements $g' - x'_i$ with $2g' = x'_1 + x'_2 + x'_3 + x'_4$, where $g - x_i = g' - x'_i$ for $i = 1,2, 3$. This last relation certainly implies that $x_1 + x'_2 = x'_1 + x_2$ and hence, by dissociativity, that $x_1 \in \{x'_1, x_2\}$. Since the elements of $Z'$ are distinct, we cannot have $x_1 = x_2$. Therefore $x_1 = x'_1$, from which it follows that $g = g'$. It then follows that $x_2 = x'_2$ and $x_3 = x'_3$. Finally, since $x'_1 + x'_2 + x'_3 + x'_4 = 2g' = 2g = x_1 + x_2 + x_3 + x_4$, it follows that $x_4 = x'_4$. But then $g - x_4 = g' - x'_4$ lies in both $Z'$ and $Z''$, contrary to assumption. This contradiction establishes the claim.
\end{proof}

\section{On a result of Alon, Krivelevich and Sudakov}\label{aks-appendix}

In this appendix we give a short proof that $\chi(\Gamma_A) \leq (2 + o(1))\frac{n}{\log_2 n}$ almost surely if $A$ is a random subset of an abelian group $G$, $|G| = n$. The argument is basically that of Alon, Krivelevich and Sudakov, but because we are dealing with Cayley sum graphs rather than arbitrary regular graphs we can use a concise Fourier argument instead of an eigenvalue argument.

We assume some familiarity with the notation of the discrete Fourier transform as discussed in \cite[Chapter 4]{tao-vu}, for example. Here we will be writing $\hat{f}(\gamma) := \E_{x \in G} f(x) \overline{\gamma(x)}$ for $\gamma \in G^*$.

\begin{proposition}\label{pseudo}
Suppose that $\sup_{\gamma \neq 1} |\hat{1}_A(\gamma)| \leq n^{-\eta}$. Then $\chi(\Gamma_A) \ll (\frac{1}{\eta} - o(1)) \frac{n}{\log_2 n}$. 
\end{proposition}
\begin{proof}
First of all note that 
\[ \sum_{x,x' \in S} 1_A(x+x') = n^2 \sum_\gamma \hat{1}_S(\gamma)^2 \hat{1}_A(\overline{\gamma}) = \frac{1}{2}|S|^2 + n^2 \sum_{\gamma \neq 1} \hat{1}_S(\gamma)^2 \hat{1}_A(\overline{\gamma}) .\]
If $S \subset G$ has $|S| > n^{1-\eta} \log n$ (say) then the error term here can be efficiently bounded by Parseval's identity and the triangle inequality:
\[ n^2\big|\sum_{\gamma \neq 1} \hat{1}_S(\gamma)^2 \hat{1}_A(\overline{\gamma})\big| \leq n^{2-\eta} \sum_\gamma |\hat{1}_S(\gamma)|^2 = n^{2-\eta}\frac{|S|}{n} = o(|S|^2).\] 
Thus \[ \sum_{x,x' \in S} 1_A(x + x') = \big(\frac{1}{2} + o(1)\big) |S|^2,\] and in particular there is some $x \in S$ such that $x + x' \notin A$ for all $x'$ in some subset $S' \subset S$ of size at least $(\frac{1}{2} - o(1))|S|$.

By repeated application of this, it follows that any set $S \subset G$ of size at least $\frac{n}{\log^2 n}$ has a subset $X$ of size at least $(\eta - o(1))\log_2 n$ with the property that $X \hat{+} X$ is disjoint from $A$.

From this the result follows straightforwardly by iteration, as in Section \ref{sec2}.
\end{proof}

If $A \subset G$ is a random set then almost surely we have
\[ \sup_{\gamma \neq 1}| \hat{1}_A(\gamma)| \ll n^{-1/2 + o(1)}.\] This follows from a standard application of Bernstein's large deviation bound for each individual $\gamma \in G^*$, followed by a union bound over all $\gamma \neq 1$; see for example \cite[Lemma 4.16]{tao-vu}. Combining this with Proposition \ref{pseudo} tells us that indeed $\chi(\Gamma_A) \ll (2 + o(1)) \frac{n}{\log_2 n}$ almost surely.

\section{Some bounds on binomial coefficients}

In this appendix we collect some bounds on binomial coefficients. These are of standard type, and we have often given crude bounds sufficient for our purposes rather than the strongest possible estimates.

\begin{lemma}\label{binom1}
Let $n \geq k \geq 1$ be integers. Then $\binom{2n}{k} \geq 2^{k} \binom{n}{k}$.
\end{lemma}
\begin{proof}
We have
\[ \frac{\binom{2n}{k}}{\binom{n}{k}}   =  \frac{2n}{n} \cdot \frac{2n - 1}{n-1} \cdots \frac{2n - k + 1}{n - k + 1} \geq 2^k,\] as required.
\end{proof}

\begin{lemma}\label{binom2}
Let $n \geq k \geq d \geq 1$ be integers. Then $\binom{n}{d} \leq k^d \binom{n}{k}^{d/k}$.
\end{lemma}
\begin{proof}
First note that $n^k \leq k^k \binom{n}{k}$. Indeed, 
\[ \binom{n}{k} = \frac{n}{k} \cdot \frac{n-1}{k-1} \dots \frac{n - k + 1}{1} \geq (\frac{n}{k})^k.\]
Therefore we have
\[ \binom{n}{d} \leq n^d = (n^k)^{d/k} \leq \big( k^k \binom{n}{k} \big)^{d/k},\] which is the stated bound.
\end{proof}

\begin{lemma}\label{binom3} Let $n \geq k \geq d \geq 1$ be integers. Then $\binom{n}{d} \leq \big(\frac{k}{n-k}\big)^{k - d} \binom{n}{k}$.
\end{lemma}
\begin{proof} We have
\[ \binom{n}{d} = \binom{n}{k} \cdot \frac{d + 1}{n - d} \cdot \frac{d+2}{n - d - 1}\dots \frac{k}{n - k +1} \leq \big(\frac{k}{n-k}\big)^{k - d}\binom{n}{k} .\]
\end{proof}

\end{document}